\documentclass[a4paper,11pt]{amsart}
\usepackage{hyperref,fancyhdr,mathrsfs,amsmath,amscd,amsthm,amsfonts,latexsym,amssymb,stmaryrd}
\usepackage[all]{xy}

\DeclareMathOperator{\rank}{rank}

\DeclareMathOperator{\dif}{d}

\newcommand{\Cal}{\mathcal{C}}
\renewcommand{\H}{\mathscr{H}}

\newcommand{\F}{\mathscr{F}}
\newcommand{\Fa}{\mathcal{F}}

\newcommand{\ol}{\mathcal{O}}

\def \a{\alpha}

\def \b{\beta}

\def \phi{\varphi}
\def \Phi{\varPhi}
\def \p{\pi}
\def \r{\rho}
\def \s{\sigma}
\def \t{\tau}

\def \R{\mathbb{R}}
\def \Hq{\mathbb{H}\,}
\def \C{\mathbb{C}\,}

\def\widecheckg{g^{\hspace*{-2.5pt}\vbox to 5pt{\hbox to
0pt{\LARGE$\check{}$}}}\hspace*{2pt}}

\def\widecheckl{\lambda^{\hspace*{-3.5pt}\vbox to 8pt{\hbox to
0pt{\LARGE$\check{}$}}}\hspace*{2pt}}

\begin{document}

\title{On the twistor space of a\\ 
(co-)CR quaternionic manifold} 
\author{Radu Pantilie} 
\email{\href{mailto:radu.pantilie@imar.ro}{radu.pantilie@imar.ro}}
\address{R.~Pantilie, Institutul de Matematic\u a ``Simion~Stoilow'' al Academiei Rom\^ane,
C.P. 1-764, 014700, Bucure\c sti, Rom\^ania}

\newtheorem{thm}{Theorem}[section]
\newtheorem{lem}[thm]{Lemma}
\newtheorem{cor}[thm]{Corollary}
\newtheorem{prop}[thm]{Proposition}

\theoremstyle{definition}

\newtheorem{defn}[thm]{Definition}
\newtheorem{rem}[thm]{Remark}
\newtheorem{exm}[thm]{Example}

\numberwithin{equation}{section}

\maketitle
\thispagestyle{empty}
\vspace{-5mm}
\section*{Abstract}
\begin{quote}
{\footnotesize 
We characterise, in the setting of the Kodaira--Spencer deformation theory, the twistor spaces of (co-)CR quaternionic manifolds. 
As an application, we prove that, locally, the leaf space of any nowhere zero quaternionic vector field on a quaternionic manifold is 
endowed with a natural co-CR quaternionic structure.\\ 
\indent 
Also, for any positive integers $k$ and $l$\,, with $kl$ even, we obtain the geometric objects whose twistorial counterparts are complex manifolds 
endowed with a conjugation without fixed points and which preserves an embedded Riemann sphere  
with normal bundle~$l\ol(k)$\,.\\ 
\indent 
We apply these results to prove the existence of natural classes of co-CR quaternionic manifolds.}
\end{quote}

\section*{Introduction} 

\indent 
Twistor Theory is based on a, nontrivial and far from being complete, dictionary between differential geometric and holomorphic objects. 
For example (see \cite{Hit-complexmfds}\,), it is well-known that (up to a conjugation) any \emph{anti-self-dual manifold} 
corresponds to a three-dimensional complex manifold endowed with a locally complete family of (embedded) Riemann spheres with normal bundle $2\ol(1)$ 
(where $\ol(1)$ is the dual of the tautological bundle over $\C\!P^1$). More generally, any \emph{quaternionic manifold} of dimension $4k$ 
corresponds to a complex manifold of dimension $2k+1$ endowed with a locally complete family of Riemann spheres with normal bundle $2k\ol(1)$\,. 
Another such natural correspondence is given by the \emph{three-dimensional Einstein--Weyl spaces} which, locally, correspond to complex 
surfaces endowed with locally complete families of Riemann spheres with normal bundle $\ol(2)$\,.\\ 
\indent 
All these correspondences involve two steps: (1) the construction of a \emph{twistor space} for each of the given differential geometric structures, and 
(2) the characterisation (among the complex manifolds endowed with a family of compact complex submanifolds) of the obtained twistor spaces.\\  
\indent 
In \cite{fq_2}\,, it is shown that step (1) of all of the above mentioned correspondences are particular cases 
of the construction by which to any \emph{co-CR quaternionic manifold} it is associated its twistor space.\\ 
\indent 
In this paper, we provide the corresponding step (2) for a large class of co-CR quaternionic manifolds, thus showing that in order for 
a complex manifold $Z$ to be the twistor space of a co-CR quaternionic manifold it is sufficient to be endowed 
with a locally-complete family $\Fa$ of Riemann spheres which (up to the restrictions imposed by a conjugation; 
see Theorem \ref{thm:charact_twist_co-cr_q}\,, below, for details) satisfies  
the following: (a) for any $t\in\Fa$, its (holomorphic) normal bundle $Nt$ is `positive' and the exact sequence 
$0\longrightarrow Tt\longrightarrow TZ|_t\longrightarrow Nt\longrightarrow 0$ splits, and (b) $\dim H^1(t,Tt\otimes N^*t)$ is independent of $t\in\Fa$. 
Note that, by \cite{fq_2}\,, condition (a) is, also, necessary for $Z$ to be the twistor space of a co-CR quaternionic manifold, with $\Fa$ the corresponding family 
of twistor lines.\\ 
\indent 
Consequently, up to a complexification, the `Veronese webs' of \cite{GelZak-91} and the `generalized hypercomplex structures' of \cite{Bie-TAMS06} 
are particular classes of co-CR quaternionic structures. Furthermore, it follows that the `bi-Hamiltonian structures' corresponding \cite{GelZak-91} 
to the former are obtained through a dimensional reduction of a hypercomplex structure.\\ 
\indent 
As an application of Theorem \ref{thm:charact_twist_co-cr_q}\,, we prove (Corollary \ref{cor:qvf}\,) that, locally, the leaf space 
of any nowhere zero quaternionic vector field on a quaternionic manifold is endowed with a natural co-CR quaternionic structure.\\ 
\indent 
We, also, provide (Theorem \ref{thm:charact_twist_cr_q}\,) the corresponding step (2) for the dually flavoured \emph{CR~quaternionic manifolds}, 
introduced in \cite{fq}\,.\\  
\indent 
Finally, a similar approach leads to (Corollary \ref{cor:qg_sR}\,) 
a natural correspondence between the following classes, where $k$ and $l$ are positive integers:\\ 
\indent 
\quad{\rm (i)} Complex manifolds endowed with a conjugation without fixed points and which preserves an embedded Riemann sphere  
with normal bundle $l\ol(k)$\,.\\ 
\indent 
\quad{\rm (ii)} Quadruples $(M,N,x,\phi)$\,, with $x\in M\subseteq N$, $N$ quaternionic, $M\subseteq N$ generic and of type $(k,l)$ 
(see Definition \ref{defn:type_cr_q}\,),  and $\phi:N\to M$ a twistorial retraction of the inclusion $M\subseteq N$.\\ 
\indent 
We apply these results to prove the existence of natural classes of co-CR quaternionic manifolds.\\ 
\indent 
I am grateful to Stefano Marchiafava for vey useful and enjoyable discussions.

\section{(Co-)CR quaternionic manifolds} 

\indent 
In this section we recall, from \cite{fq} and \cite{fq_2} to which we refer for further details, the notions of CR and co-CR quaternionic manifolds.\\ 
\indent 
Let $\Hq$ be the (unital) associative algebra of quaternions; note that, its automorphism group is ${\rm SO}(3,\R)$ acting trivially on $\R$ and canonically on 
${\rm Im}\Hq(=\R^3)$\,.\\ 
\indent 
A \emph{linear quaternionic structure} on a (real) vector space $E$ is an equivalence class of morphisms of associative algebras from $\Hq$ to ${\rm End}(E)$\,, 
where to such morphisms $\s_1$ and $\s_2$ are equivalent if $\s_2=\s_1\circ a$\,, for some $a\in{\rm SO}(3,\R)$\,.\\ 
\indent 
Let $E$ be a quaternionic vector space whose structure is given by the morphism $\s:\Hq\to{\rm End}(E)$\,. Then $Z=\s(S^2)$ 
depends only of the equivalence class of $\s$\,. Moreover, $Z$ determines the linear quaternionic structure of $E$. Also, any $J\in Z$ 
is a linear complex structure on $E$ which is called \emph{admissible} (for the given linear quaternionic structure).\\ 
\indent 
A \emph{linear CR quaternionic structure} on a vector space $U$ is a pair $(E,\iota)$\,, where $E$ is a quaternionic vector space and 
$\iota:U\to E$ is an injective linear map such that, for any admissible linear complex structure $J$ on $E$, we have 
${\rm im}\,\iota+J({\rm im}\,\iota)=E$.\\ 
\indent 
By duality we obtain the notion of \emph{linear co-CR quaternionic structure}.\\ 
\indent 
Any quaternionic vector space $E$ is isomorphic to $\Hq^{\!k}$, where $\dim E=4k$\,, and the automorphism group of $\Hq^{\!k}$ is ${\rm Sp}(1)\cdot{\rm GL}(k,\Hq)$\,.\\ 
\indent 
The classification of the (co-)CR quaternionic vector spaces is much less trivial. It is based on a covariant functor from the category of CR quaternionic vector spaces 
to the category of holomorphic vector bundles over the Riemann spheres, which we next describe. 
Let $(U,E,\iota)$ be a CR quaternionic vector space and let $Z$ be the space of admissible linear complex structures on $E$. For any $J\in Z$ 
denote $E^J={\rm ker}(J+{\rm i})$\,. Then $E^{0,1}=\bigcup_{J\in Z}\{J\}\times E^J$ is a holomorphic vector subbundle of $Z\times E^{\C}$ 
(isomorphic to $2k\ol(-1)$\,, where $\dim E=4k$ and $\ol(-1)$ is the tautological line bundle on $Z=\C\!P^1$). Furthermore, $\iota^{-1}\bigl(E^{0,1}\bigr)$ 
is a holomorphic vector subbundle of $Z\times U^{\C}$ which is called \emph{the holomorphic vector bundle of $(U,E,\iota)$} \cite{fq}\,. 
It follows that there exists a natural bijective correspondence between (isomorphism classes of) CR quaternionic vector spaces and holomorphic vector bundles, 
over the Riemann sphere, whose Birkhoff--Grothendieck decompositions contain only terms whose Chern numbers are at most $-1$\,.\\ 
\indent 
A \emph{quaternionic vector bundle} is a real vector bundle $E$ with typical fibre $\Hq^{\!k}$ and structural group ${\rm Sp}(1)\cdot{\rm GL}(k,\Hq)$\,. 
If $E$ is a quaternionic vector bundle then the space $Z$ of admissible linear complex structures on $E$ is a (locally trivial) fibre bundle with typical fibre $S^2$ 
and structural group ${\rm SO}(3,\R)$\,.\\ 
\indent 
An \emph{almost CR quaternionic structure} on a manifold $M$ is a pair $(E,\iota)$\,, where $E$ is a quaternionic vector bundle over $M$ and 
$\iota:TM\to E$ is a vector bundles morphism such that $(E_x,\iota_x)$ defines a linear CR quaternionic structure on $T_xM$, for any $x\in M$.\\ 
\indent 
By duality we obtain the notion of \emph{almost co-CR quaternionic structure}.\\ 
\indent 
Let $(M,E,\iota)$ be an almost CR quaternionic manifold. Suppose that $E$ is endowed with a compatible connection $\nabla$ and let 
$\p:Z\to M$ be the bundle of admissible linear complex structures on $E$. Denote by $\mathcal{B}$ the complex distribution on $Z$ 
whose fibre, at each $J\in Z$, is the horizontal lift, with respect to $\nabla$, of $\iota^{-1}\bigl({\rm ker}(J+{\rm i})\bigr)$\,. 
Then $\Cal=({\rm ker}\dif\!\p)^{0,1}\oplus\mathcal{B}$ is an almost CR structure on $Z$. If $\Cal$ is integrable then 
$(E,\iota,\nabla)$ is a \emph{CR quaternionic structure} on $M$.\\ 
\indent 
A \emph{quaternionic manifold} is a CR quaternionic manifold $(M,E,\iota,\nabla)$ for which $\iota$ is an isomorphism.\\ 
\indent 
Let $N$ be a quaternionic manifold. A submanifold $M\subseteq N$ is \emph{generic} if $(TN|_M,\iota)$ is an almost CR quaternionic structure on $M$, 
where $\iota:TM\to TN|_M$ is the inclusion; in particular, ${\rm codim}\,M\leq 2k-1$, where $\dim N=4k$\,. The terminology is justified by the fact that the 
set of real vector subspaces $U\subseteq\Hq^{\!k}$, of fixed codimension $l\leq2k-1$ and on which $\Hq^{\!k}$ induces a linear CR quaternionic structure, 
is open in the Grassmannian of real vector subspaces of codimension $l$ of $\Hq^{\!k}$. Also, $M\subseteq N$ is generic if and only if, 
for any admissible local complex structure $J$ on $N$, we have that $M$ is a generic CR~submanifold of $(N,J)$\,.\\ 
\indent 
Any hypersurface of a quaternionic manifold is generic, but this does not hold for higher codimensions as simple examples show 
(take, for example, $\Hq^{\!k}\times\C$ in $\Hq^{\!k+1}$).\\ 
\indent 
Any generic submanifold of a quaternionic manifold inherits a natural CR quaternionic structure. Conversely, 
any real analytic CR quaternionic structure is obtained this way from a germ unique embedding into a quaternionic manifold.\\ 
\indent 
Let $(M,E,\r)$ be an almost co-CR quaternionic manifold. Suppose that $E$ is endowed with a compatible connection $\nabla$ and let 
$\p:Z\to M$ be the bundle of admissible linear complex structures on $E$. Denote by $\mathcal{B}$ the complex distribution on $Z$ 
whose fibre, at each $J\in Z$, is the horizontal lift, with respect to $\nabla$, of $\r\bigl({\rm ker}(J+{\rm i})\bigr)$\,. 
Then $\Cal=({\rm ker}\dif\!\p)^{0,1}\oplus\mathcal{B}$ is an almost co-CR structure on $Z$; that is, $\Cal+\overline{\Cal}=T^{\C\!}Z$. 
If $\Cal$ is integrable (that is, its space of sections is closed under the usual bracket) then 
$(E,\iota,\r)$ is a \emph{co-CR quaternionic structure} on $M$. Suppose, further, that $\Cal\cap\overline{\Cal}=({\rm ker}\dif\!\p_Y)^{\C}$, 
where $\p_Y:Z\to Y$  is a surjective submersion whose restriction to each fibre of $\p$ is injective, and with respect to which $\Cal$ is projectable 
(note that, the last condition is automatically satisfied if $\p_Y$ has connected fibres). Then $Y$ endowed with $\dif\!\p_Y(\Cal)$  
is a complex manifold, called the \emph{twistor space} of $(M,E,\r,\nabla)$\,; in particular, $\p_Y$ restricted to each fibre of $\p$ 
is an injective holomorphic immersion.

\section{On the twistor space of a CR quaternionic manifold}

\indent 
Recall \cite{KodSpe-I_II} that a \emph{smooth family of complex manifolds} is a surjective submersion $\p:Z\to M$ whose 
domain is endowed with a CR structure $\mathcal{D}$ such that $\mathcal{D}\oplus\overline{\mathcal{D}}=({\rm ker}\dif\!\p)^{\C\!}$. 
Any member of the family is a fibre $t$ of $Z$ endowed with the complex structure for which $T^{0,1}t=\mathcal{D}|_t$\,.\\ 
\indent 
We shall denote by $\mathcal{S}$ the generating subsheaf of $T^{\C\!}Z/({\rm ker}\dif\!\p)^{0,1}$ formed 
of the complex vector fields on $Z$ which are projectable with respect to $\p$ and holomorphic when restricted to the fibres of $Z$. 
Note that,  the exact sequence  
$$0 \longrightarrow ({\rm ker}\dif\!\p)^{1,0} \longrightarrow \mathcal{S} \longrightarrow \mathcal \p^*(T^{\C\!}M) \longrightarrow 0$$
is the \emph{fundamental sequence} \cite[(4.1)]{KodSpe-I_II}\,. 
 
\begin{thm} \label{thm:charact_twist_cr_q}
Let $\p:Z\to M$ be a smooth family of Riemann spheres.  
Suppose that $Z$ is endowed with a CR structure $\Cal$ and an involutive diffeomorphism $\t$ such that:\\ 
\indent 
\quad{\rm (i)} $\Cal$ induces the given complex structure on each fibre of $Z$\,;\\ 
\indent 
\quad{\rm (ii)} For each fibre $t$ of $Z$\,, we have that $(\Cal|_t)/T^{0,1}t$ is a holomorphic subbundle 
of the restriction to $t$ of $T^{\C\!}Z/({\rm ker}\dif\!\p)^{0,1}$ such that the induced quotient of $\bigl(T^{\C\!}Z/\Cal\bigr)|_t$ through $T^{1,0}t$ 
is isomorphic to $k\ol(1)$\,, for some positive integer $k$\,;\\ 
\indent 
\quad{\rm (iii)} $\t$ is anti-CR with respect to $\Cal$, preserves the fibres of $Z$ and has no fixed points.\\ 
\indent 
Then there exists a CR quaternionic structure $(E,\iota,\nabla)$ on $M$ whose twistor space is $(Z,\Cal)$\,; 
moreover, $(E,\iota)$ is unique (up to isomorphisms) with these properties.
\end{thm} 
\begin{proof}  
We have that $Z$ is a fibre bundle with typical fibre the Riemann sphere (apply \cite[Theorem 6.3]{KodSpe-I_II}\,). 
Furthermore, $\t$ restricted to each fibre of $Z$ is an involutive conjugation without fixed points and, thus, it is the antipodal map. 
Therefore $Z$ is a sphere bundle with structural group ${\rm SO}(3)$\,.\\ 
\indent  
Denote $\mathcal{B}=\Cal/({\rm ker}\dif\!\p)^{0,1}$ and note that by using the fundamental sequence and (ii) we obtain an exact sequence 
\begin{equation} \label{e:sequence_almost_cr_q} 
0 \longrightarrow \mathcal{B} \longrightarrow \p^*(T^{\C\!}M)  \longrightarrow \p^*(T^{\C\!}M)/\mathcal{B} \longrightarrow 0
\end{equation} 
of complex vector bundles which are holomorphic when restricted to the fibres of $Z$. 
Furthermore, (ii)\,, \cite[Proposition 3.3]{Qui-QJM98}\,, and \cite[\S3]{fq} implies that, at each $x\in M$, 
there exists a unique linear CR quaternionic structure $(E_x,\iota_x)$ on $T_xM$ 
whose holomorphic vector bundle is the restriction of $\mathcal{B}$ to $\p^{-1}(x)$\,.\\ 
\indent 
To describe $(E,\iota)$ we use \eqref{e:sequence_almost_cr_q}\,. Firstly, $E^{\C\!}$ is the direct image through $\p$ of $\p^*(T^{\C\!}M)/\mathcal{B}$ 
(that is, for any open subset $U\subseteq M$, the space of sections of $E^{\C\!}|_U$ is the space of sections 
(which are holomorphic when restricted to the fibres of $\p$) 
of the restriction to $\p^{-1}(U)$ of $\p^*(T^{\C\!}M)/\mathcal{B}$\,; 
the fact that $E^{\C}$ is a bundle is given by \cite[Theorem 9]{KodSpe-III}\,). 
Further, the long exact sequence of cohomology of \eqref{e:sequence_almost_cr_q} gives an injective complex linear map $\a$ between the spaces of sections of 
$\p^*(T^{\C\!}M)$ and $\p^*(T^{\C\!}M)/\mathcal{B}$\,, over any open set of the form $\p^{-1}(U)$\,, where $U$ is an open subset of $M$. 
Then $\iota|_U$ is the restriction of $\a$ to the space of sections of $\p^*(T^{\C\!}M)|_U$ which 
intertwine the conjugation and the antipodal map.\\ 
\indent 
Now, any connection on $Z$ corresponds to a splitting of the fundamental sequence (cf.\ \cite[Proposition 5.1]{KodSpe-I_II}\,), 
which is invariant under the antipodal map. Obviously, such splittings exist but we want to obtain 
$\mathcal{S}=\H\oplus({\rm ker}\dif\!\p)^{1,0}$ with $\mathcal{B}\subseteq\H$\,. 
To prove this, note that, instead of the fundamental sequence, we may use the exact sequence 
\begin{equation} \label{e:connection_sequence} 
0 \longrightarrow ({\rm ker}\dif\!\p)^{1,0} \longrightarrow T^{\C\!}Z/({\rm ker}\dif\!\p)^{0,1} \longrightarrow \mathcal \p^*(T^{\C\!}M) \longrightarrow 0\;, 
\end{equation} 
whilst, any splitting of \eqref{e:connection_sequence} whose `image' contains $\mathcal{B}$ corresponds to a splitting of 
\begin{equation} \label{e:connection_sequence_B}
0 \longrightarrow ({\rm ker}\dif\!\p)^{1,0} \longrightarrow  T^{\C\!}Z/\Cal \longrightarrow \p^*(T^{\C\!}M)/\mathcal{B} \longrightarrow 0\;. 
\end{equation} 
But the restriction of \eqref{e:connection_sequence_B} to each fibre $t\,(=\C\!P^1)$ of $Z$ is 
$$0 \longrightarrow \ol(2) \longrightarrow \bigl( T^{\C\!}Z/\Cal\bigr)|_t \longrightarrow 2k\ol(1) \longrightarrow 0\;,$$   
where $\rank E=4k$\,. Together with \cite[Theorem 10]{KodSpe-III}\,,  this quickly implies that there exists a connection on $Z$ whose complexification 
contains $\mathcal{B}$\,, and the proof follows. 
\end{proof} 

\indent 
Note that, by \cite{fq}\,, the twistor space of any CR quaternionic manifold satisfies the conditions of Theorem \ref{thm:charact_twist_cr_q}\,.\\ 
\indent 
Also, in Theorem \ref{thm:charact_twist_cr_q}\,, we have that $\Cal$ is a complex structure on $Z$ if and only if $\dim M=4k$\,. 
Then (i) and (ii) are equivalent to the fact that any fibre of $Z$ is a complex submanifold and its (holomorphic) normal bundle 
is isomorphic to $k\ol(1)$\,. Thus, Theorem \ref{thm:charact_twist_cr_q} gives, in particular, the classical characterisation 
of the twistor space of any quaternionic manifold.

\section{On the twistor space of a co-CR quaternionic manifold}

\indent 
Recall \cite{KodSpe-I_II} (see \cite{Kod}\,) that a family $\Fa$ of compact complex submanifolds of a complex manifold $Z$ is \emph{complex analytic} if 
there exist complex manifolds $P$ and $Q$ and holomorphic maps $\p_Z:Q\to Z$\,, $\p:Q\to P$, with $\p$ a proper surjective submersion, such that  
$\Fa=\bigl\{\p_Z(\p^{-1}(x))\bigr\}_{x\in P}$ and $\p_Z$ restricted to each fibre of $\p$ is an injective immersion. 

\begin{thm} \label{thm:charact_twist_co-cr_q} 
Let $Z$ be a complex manifold endowed with a conjugation $\t$, without fixed points, and a locally complete family $\Fa$ of Riemann spheres.\\ 
\indent 
Then, locally, $Z$ is the twistor space of a co-CR quaternionic manifold, for which $\Fa$ is the family of twistor lines, if the following conditions are satisfied:\\ 
\indent
\quad{\rm (i)} The Birkhoff--Grothendieck decomposition of the normal bundle $Nt$\,, of any $t\in\Fa$\,, contains only terms of Chern number at least $1$\,;\\ 
\indent
\quad{\rm (ii)} The exact sequence $0\longrightarrow Tt \longrightarrow TZ|_t \longrightarrow Nt \longrightarrow 0$ splits, for any $t\in \Fa$, 
and $\dim H^1(t,Tt\otimes N^*t)$ is independent of $t\in \Fa$, 
where $Tt$ and $TZ$ are the holomorphic tangent bundles of $t$ and $Z$, respectively;\\ 
\indent
\quad{\rm (iii)} $\t(\Fa)=\Fa$ and there exists $t_0\in\Fa$ such that $\t(t_0)=t_0$\,.  
\end{thm} 
\begin{proof}  
Unless otherwise stated, all the objects and maps are assumed complex analytic; in particular, if $P$ is a (complex) manifold then $TP$ denotes 
its holomorphic tangent bundle.\\ 
\indent 
By (i)\,, \cite{Kod}\,, and \cite[Theorem 6.3]{KodSpe-I_II}\,, the family $\Fa$ is given by a map $\p_Z:Q\to Z$, 
where $\p:Q\to P$ is a locally trivial fibre bundle with typical fibre $\C\!P^1$. Also, for any $x\in P$, there exists a natural isomorphism $T_xP=H^0(t_x,Nt_x)$\,, 
where $t_x=\p_Z(\p^{-1}(x))$\,. Furthermore, from (i) it follows quickly that $\p_Z$ is a submersion which, by passing to an open subset of $Z$, 
can be assumed surjective.\\ 
\indent  
Let $\mathcal{B}={\rm ker}\dif\!\p_Z$\,. {}From the isomorphism between $\p^*(TP)$ and the quotient of $TQ$ through ${\rm ker}\dif\!\p$\,, 
we obtain 
\begin{equation} \label{e:charact_twist_co-cr_q_2}
0 \longrightarrow {\rm ker}\dif\!\p \longrightarrow  TQ/\mathcal{B} \longrightarrow \p^*(TP)/\mathcal{B} \longrightarrow 0 
\end{equation} 
which, for any $x\in P$,  restricts to 
\begin{equation} \label{e:charact_twist_co-cr_q_3}
0\longrightarrow Tt_x \longrightarrow TZ|_{t_x} \longrightarrow Nt_x \longrightarrow 0\;, 
\end{equation} 
and, consequently, gives 
\begin{equation} \label{e:charact_twist_co-cr_q_1} 
0 \longrightarrow \mathcal{B}|_{t_x} \longrightarrow t_x\times T_xM \to Nt_x \longrightarrow 0\;, 
\end{equation} 
where we have identified $t_x=\p^{-1}(x)$\,.\\ 
\indent 
Condition (iii) gives that by passing, if necessary, to an open neighborhood of each point of $P$ corresponding to a $\t$-invariant member of $\F$, 
we may suppose $P$ be the complexification of a real-analytic submanifold $M$.\\ 
\indent 
Now, similarly to the proof of Theorem \ref{thm:charact_twist_cr_q}\,, we obtain that $M$ is endowed with an almost co-CR quaternionic structure 
for which $Q|_M$ is the bundle of admissible linear complex structures. Furthermore, any connection on $Q|_M$ containing $\mathcal{B}|_{Q|_M}$ 
corresponds to a splitting of the restriction to $Q|_M$ of \eqref{e:charact_twist_co-cr_q_2} which by (ii) and \cite[Theorem 10]{KodSpe-III} exists, 
and the proof follows. 
\end{proof} 

\indent 
Theorem \ref{thm:charact_twist_co-cr_q} can be slightly extended, as follows. 

\begin{rem} \label{rem:charact_co-cr_q_extended}  
1) If in (i) of Theorem \ref{thm:charact_twist_co-cr_q} we further assume that the Chern numbers are contained by $\{1,2,3\}$ 
then condition (ii) becomes superfluous.\\ 
\indent 
2) The conclusion of Theorem \ref{thm:charact_twist_co-cr_q} still holds if we assume that $Z$ contains a Riemann sphere $t$ 
preserved by $\t$, which satisfies (ii)\,, and (i)\,, where for the latter the Chern numbers are contained by $\{k,k+1\}$\,, for some positive integer $k$ 
(this is an immediate consequence of \cite{Kod}\,, \cite[Theorems 6.3\,, 7.4]{KodSpe-I_II}\,, and Theorem \ref{thm:charact_twist_co-cr_q}\,). 
\end{rem} 

\indent  
Remark \ref{rem:charact_co-cr_q_extended} shows, in particular, that Theorem \ref{thm:charact_twist_co-cr_q} is a natural generalization of classical results 
(see \cite{Hit-complexmfds}\,, \cite{IMOP} and the references therein) on quaternionic and anti-self-dual manifolds, and three-dimensional Einstein--Weyl spaces.\\ 
\indent 
It is well known that, locally, the leaf space of a nowhere zero conformal vector field, on an anti-self-dual manifold, is a (three-dimensional) Einstein--Weyl space 
(see \cite{PanWoo-sd} and the references therein). In higher dimensions, we have the following result. 

\begin{cor} \label{cor:qvf}
Locally, the leaf space of any nowhere zero quaternionic vector field on a quaternionic manifold is a co-CR quaternionic manifold. 
\end{cor} 
\begin{proof} 
Let $V$ be a nowhere zero quaternionic vector field on a quaternionic manifold $M$ whose orbits are the fibres of a submersion $\phi:M\to N$. 
Then $V$ lifts to a holomorphic vector field $\widetilde{V}$ on the twistor space $Z_M$ of $M$ (use, for example, \cite{IMOP}\,). 
Assume, for simplicity, that the one-dimensional holomorphic foliation generated by $\widetilde{V}$ is simple; that is, it is given by the fibers 
of a holomorphic submersion $\Phi$ from $Z_M$ onto some complex manifold $Z_N$\,.\\ 
\indent 
Then $\Phi$ maps the twistor lines on $Z_M$ onto a complex analytic family of Riemann spheres on $Z_N$ each of which has normal bundle 
$2k\ol(1)\oplus\ol(2)$\,, where $\dim M=4(k+1)$\,. Furthermore, this family is parametrised by a complexification of $N$; 
consequently, it is locally complete (apply \cite{Kod}\,).\\ 
\indent 
{}By Remark \ref{rem:charact_co-cr_q_extended}(1)\,, we have that $Z_N$ is the twistor space of a co-CR quaternionic structure on $N$. 
\end{proof} 
 
\indent 
Note that, in the proof of Corollary \ref{cor:qvf}\,, the induced twistorial structure (of the co-CR quaternionic structure) on $N$ is unique 
with the property that $\phi$ be a twistorial map (see \cite{PanWoo-sd} and \cite{LouPan-II} for the definition of twistorial structures and maps).\\ 
\indent 
On the other hand, unlike the complex setting, the quaternionic distribution generated by a quaternionic vector field is not necessarily integrable. 
Indeed, if not, then any homogeneous quaternionic manifold would be locally isomorphic with the quaternionic projective space - a contradiction.

\section{Another natural twistorial correspondence} 

\indent 
Recall (see \cite[p.\ 172]{GolGui} and note that the 
definitions extend easily to the complex analytic category) that two complex curves $c_1$ and $c_2$ on a complex manifold $Z$ 
\emph{have a contact of order $k$} at a point $x\in c_1\cap c_2$ if for any holomorphic function $f$ defined on some open neighborhood of $x$ in $Z$ 
we have that $f|_{c_1}$ vanishes up to the $k$-th order at $x$ if and only if $f|_{c_2}$ vanishes up to the $k$-th order at $x$\,. Then a \emph{$k$-jet of (complex) curves} 
on $Z$ at $x$ is an equivalence class of curves on $Z$ which have a contact of order $k$ at $x$\,. Further, the set $Y_k(Z)$ of all $k$-jets of 
curves at all the points of $Z$ is, in a natural way, a complex manifold of dimension $kl+l+1$\,, where $\dim_{\C\!}Z=l+1$\,. Moreover, the canonical map from $Y_k(Z)$ onto $Z$ 
is the projection of a locally trivial fibre space; in particular, $Y_0(Z)=Z$, whilst $Y_1(Z)$ is the projectivisation of the holomorphic tangent bundle of $Z$.

\begin{thm} \label{thm:q_mfds_from_rational_surf}
Let $Z$ be a complex manifold endowed with a conjugation without fixed points and which preserves an embedded Riemann sphere $t\subseteq Z$ 
whose normal bundle is $l\,\ol(k+1)$\,, with $k$ and $l$ positive integers.\\ 
\indent 
Then there exists a quaternionic manifold, of dimension $2l(k+1)$\,, whose twistor space is an open subset of $Y_k(Z)$\,, endowed with the conjugation induced by $\t$, 
and for which the canonical lift of $t$ to $Y_k(Z)$ is a twistor line. 
\end{thm} 
\begin{proof} 
The conjugation on $Z$ induces a conjugation on the normal bundle of the embedded Riemann sphere, covering the antipodal map. 
Hence, if $k$ is even then, also, $l$ must be even (see \cite{Qui-QJM98}\,).\\ 
\indent 
Then, as in the proof of Theorem \ref{thm:charact_twist_co-cr_q}\,, we obtain on $Z$ a locally complete family $\Fa$ of Riemann spheres 
given by holomorphic maps $\p_Z:Q\to Z$ and $\p:Q\to P$, where the former can be assumed a surjective submersion, whilst the latter 
is a locally trivial fibre bundle with typical fibre $\C\!P^1$. Also, by passing to an open subset, if necessary, $P$ is the complexification of a 
real analytic manifold $M$.\\ 
\indent 
Now, if we suitably blow up $k+1$ times $Z$ at any point $z\in u$\,, of any $u\in\Fa$, we obtain a complex manifold $Z_{u,z}$ endowed with an embedded 
Riemann sphere with trivial normal bundle (of rank $l$\,). Hence, this is contained in an $l$-dimensional locally complete family $\Fa_{u,z}$ of Riemann spheres, 
embedded in $Z_{u,z}$\,,  all of which are obtained as proper transforms of members of $\Fa$. Furthermore, $v\in\Fa$ transforms to a member of $\Fa_{u,z}$ if and only if 
$u$ and $v$ have a contact of order $k$ at $z$\,.\\ 
\indent 
Therefore $\p_Z$ factors into a holomorphic submersion with $l$-dimensional fibres from $Q$ to $Y_k(Z)$ followed by the projection from $Y_k(Z)$ onto $Z$. 
Moreover, the normal bundle of the canonical lift to $Y_k(Z)$ of any member of $\Fa$ is isomorphic to $l(k+1)\ol(1)$\,. This shows that an open subset of $Y_k(Z)$ 
is the twistor space of a quaternionic manifold $N$. Moreover, the projection from $Y_k(Z)$ onto $Z$ corresponds to a twistorial retraction of the inclusion $M\subseteq N$ 
whose differential at each point is given by the cohomology sequence of the exact sequence 
$0\longrightarrow kl\ol\longrightarrow l(k+1)\ol(1)\longrightarrow l\ol(k+1)\longrightarrow0$\,; 
in particular, $M$ is a generic submanifold of $N$ (note that, the induced almost CR quaternionic structure on $M$ can be also obtained 
by using the proof of Theorem \ref{thm:charact_twist_cr_q}\,, with $Z$ replaced by $Q|_M$ endowed with the CR structure induced from $Q$\,). 
\end{proof} 

\indent 
We say (compare \cite{GolGui}\,) that two embeddings $\phi:P\to Q$ and $\psi:P\to R$ define the same \emph{embedding germ} 
if there exist open neighbourhoods $U$ and $V$ of $\phi(P)$ and $\psi(P)$\,, respectively, and a diffeomorphism $\xi:U\to V$ such that 
$\psi=\xi\circ\phi$\,; certainly, if $Q$ and $R$ are endowed with some geometric structure then we require $\xi$ to preserve it. 

\begin{defn} \label{defn:type_cr_q} 
Let $(M,E,\iota)$ be an almost CR quaternionic manifold and let $k$ and $l$ be positive integers. We say that $(M,E,\iota)$ is of 
\emph{type $(k,l)$} if, for any $x\in M$, the holomorphic vector bundle of $(T_xM,E_x,\iota_x)$ is $l\ol(-k)$\,. 
\end{defn} 

\indent 
Note that, if $N$ is quaternionic and $M\subseteq N$ is generic of type $(k,l)$ then $\dim M=l(k+1)$ and $\dim N=2kl$\,; in particular, 
if $\dim N=2(\dim M-1)$ then the type of $M$ is determined by its dimension. 

\begin{cor} \label{cor:qg_sR} 
There exists a natural correspondence between the following classes, where $k$ and $l$ are positive integers:\\ 
\indent 
\quad{\rm (i)} Complex manifolds endowed with a conjugation without fixed points and which preserves an embedded Riemann sphere  
with normal bundle $l\ol(k)$\,.\\ 
\indent 
\quad{\rm (ii)} Quadruples $(M,N,x,\phi)$\,, with $x\in M\subseteq N$, $N$ quaternionic, $M\subseteq N$ generic and of type $(k,l)$\,,  
and $\phi:N\to M$ a twistorial retraction of the inclusion $M\subseteq N$.\\ 
\indent 
Moreover, the correspondence is bijective if we pass to embedding germs.  
\end{cor} 
\begin{proof} 
How to pass from objects as in (i) to objects as in (ii) it is shown in the proof of Theorem \ref{thm:q_mfds_from_rational_surf}\,.\\ 
\indent  
Conversely, given a quadruple $(M,N,x,\phi)$ as in (ii) let 
$Z(N)$ be the twistor space of $N$. Then $\phi$ corresponds to a holomorphic submersion $\Phi$ from $Z(N)$ onto some complex manifold $Z$ such that 
the family of twistor lines on $Z(N)$ is mapped by $\Phi$ into a family of Riemann spheres embedded into $Z$. Consequently, each member of this family 
has normal bundle isomorphic to $l\ol(k)$\,. Furthermore, as $\phi$ is a retraction of the inclusion $M\subseteq N$, at least locally, the parameter space of this family 
is a complexification of $M$ and therefore $Z$ is also endowed with a conjugation $\t$. 
Now, let $t\subseteq Z$ be the image through $\Phi$ of the twistor line corresponding to $x$\,. Then $(Z,\t,t)$ satisfies condition~(i)\,. 
\end{proof} 

\indent 
If in Theorem \ref{thm:q_mfds_from_rational_surf}\,, we have $k=l=1$\,, and if in Corollary \ref{cor:qg_sR}\,, we have $k=l+1=2$\,, 
then we obtain results of \cite{Hit-complexmfds}\,.
 
\begin{exm} \label{exm:rrs} 
Let $n$ be a nonnegative even number and let $S_n$ be the projectivisation of $\ol\oplus\ol(n)$\,. Then the conjugations of $\ol$ and $\ol(n)$ 
induce a conjugation $\t$ of $S_n$ covering the antipodal map.\\ 
\indent 
Any meromorphic section of $\ol(n)$ corresponds, up to a constant nonzero factor, to a divisor on $\C\!P^1$ 
(with $m$ poles and $n+m$ zeros, for some natural number $m$\,). Furthermore, if the divisor 
is invariant under the antipodal map then the corresponding meromorphic section will intertwine the antipodal map and the conjugation of $\ol(n)$\,. 
Let $s$ be such a meromorphic section of $\ol(n)$ with $m$ poles (necessarily, $m$ is even). Then the closure $t$ of the image of the section of $S_n$ induced by $(1,s)$ 
is preserved by $\t$\,. Also, the Chern number of its normal bundle is equal to $n+2m$ (see \cite[p.~519]{GriHar}\,). Thus, $(S_n,\t,t)$ satisfies (i) of Corollary \ref{cor:qg_sR}\,, 
with $k=n+2m$ and $l=1$. 
\end{exm} 

\indent 
If $n=0$\,, in Example \ref{exm:rrs}\,, then $m$ can, also, be any odd natural number, by suitably changing the conjugation: 

\begin{exm} \label{exm:rms} 
Let $\s$ be the conjugation on $S_0=\C\!P^1\times\C\!P^1$ given by the antipodal map, acting on each factor, and let $Y_m$ be the 
space of $(2m-1)$-jets of maps from $\C\!P^1$ to itself, where $m$ is an odd natural number. 
On denoting by $\a$ and $\b$ the source and target projections, respectively, from $Y_m$ onto $\C\!P^1$ 
then $(\a,\b):Y_m\to\C\!P^1\times\C\!P^1$ is the projection of a locally trivial fibre space with typical fibre the vector space of polynomials, 
in one (complex) variable, of degree at most $2m-1$ and with zero constant term; in particular, $\dim_{\C\!}Y_m=2m+1$ (cf.~\cite{GolGui}\,).\\ 
\indent 
On associating to each jet $[\phi]\in Y_m$\,, at $(x,y)\in\C\!P^1\times\C\!P^1$, the $(2m-1)$-jet of curves on $\C\!P^1\times\C\!P^1$, at $(x,y)$\,, 
given by the graph of $\phi$\,, we see that $Y_m$ is an open subset of the space of $(2m-1)$-jets of curves on $\C\!P^1\times\C\!P^1$. Therefore 
$Y_m$ is the twistor space of a quaternionic manifold $N_m$ of dimension $4m$\,. Note that, the twistor lines on $Y_m$ are images of 
sections $s$ of $\a$ such that $\b\circ s:\C\!P^1\to\C\!P^1$ has degree $m$\,. 
Furthermore, the corresponding generic submanifold $M_m\subseteq N_m$\,, of (ii) of Corollary \ref{cor:qg_sR}\,, is just the space of holomorphic maps of degree $m$\,, 
from $\C\!P^1$ to itself, which commute with the antipodal map.   
\end{exm} 

\indent 
Note that, for all the generic submanifolds of (ii) of Corollary \ref{cor:qg_sR}\,, given by Examples \ref{exm:rrs} and \ref{exm:rms}\,, 
condition (ii) of Theorem \ref{thm:charact_twist_co-cr_q} is automatically satisfied such that the corresponding rational ruled surfaces 
are the twistor spaces of co-CR quaternionic manifolds; in fact, hyper co-CR manifolds (see \cite{fq_2} for the definition of the latter). 
Therefore suitable products of these manifolds provide examples covering all possible $k$ and $l$ in Corollary \ref{cor:qg_sR}\,. 

\begin{exm}  \label{exm:quadric} 
Let $n\in\mathbb{N}\setminus\{0\}$ and let $\R^{n+3}\subseteq\R^{n+4}$ be embedded as a vector subspace; denote by $\ell\subseteq\R^{n+4}$ 
the orthogonal complement of $\R^{n+3}$, oriented so that the isomorphism $\R^{n+4}=\ell\oplus\R^{n+3}$ be orientation preserving.\\ 
\indent 
Let $M$ be the Grassmannian of three-dimensional oriented subspaces of $\R^{n+3}$, and let $N$ be the Grassmannian of four-dimensional oriented subspaces 
of $\R^{n+4}$ which are not contained by $\R^{n+3}$. It is well known that $N$ is a quaternionic manifold (it is an open subset 
of a Wolf space). Also, $M$ is both a CR quaternionic and a co-CR quaternionic manifold \cite{fq_2}\,. 
We embedd $M\subseteq N$ by $p\mapsto\ell\oplus p$\,, and define its retraction $\phi:N\to M$, $q\mapsto q\cap\R^{n+3}$, where $q\cap\R^{n+3}$ 
is oriented so that the orthogonal decomposition $q=\ell_q\oplus\bigl(q\cap\R^{n+3}\bigr)$ be orientation preserving, where 
$\ell_q=q\cap\bigl(q\cap\R^{n+3}\bigr)^{\perp}$, 
oriented so that its positive unit vector is characterised by the fact that its scalar product with the positive unit vector of $\ell$ be positive. 
To show that $\phi$ is twistorial, let $\C^{\!n+4}$ be the complexification of  $\R^{n+4}$. Then $Z(M)\subseteq\C\!P^{n+2}$ is the hyperquadric of isotropic directions 
in $\C^{\!n+3}$, whilst $Z(N)$ is the space of  two-dimensional isotropic (complex) vector subspaces of $\C^{\!n+4}$ which are not contained by $\C^{\!n+3}$. 
Then $\phi$ corresponds to the holomorphic map $\Phi:Z(N)\to Z(M)$\,, $q\mapsto q\cap\C^{\!n+3}$.\\ 
\indent 
Accordingly, the normal bundle of a twistor sphere in $Z(M)$ is $n\ol(2)$\,.\\ 
\indent 
Finally, note that $M$ is a CR quaternionic submanifold, but not a co-CR quaternionic submanifold of $N$.  
\end{exm}

\end{document}